\documentclass[10pt]{amsart}

\usepackage{latexsym}
\usepackage{amssymb}
\usepackage{amsmath}

\newtheorem{theorem}{Theorem}[section]

\newtheorem{proposition}[theorem]{Proposition}

\theoremstyle{definition}
\newtheorem{definition}[theorem]{Definition}

\theoremstyle{remark}
\newtheorem{remark}[theorem]{Remark}

\newcommand{\N}{\mathbb{N}}

\newcommand{\B}{\mathcal{B}}
\newcommand{\F}{\mathcal{F}}
\newcommand{\G}{\mathcal{G}}
\newcommand{\U}{\mathcal{U}}
\newcommand{\V}{\mathcal{V}}

\begin{document}

\title{Translation invariant filters and 
\\
van der Waerden's Theorem}

\author{Mauro Di Nasso}

\address{Dipartimento di Matematica\\
Universit\`a di Pisa, Italy}

\email{mauro.di.nasso@unipi.it}



\begin{abstract}
We present a self-contained proof of a strong version of 
van der Waerden's Theorem. By using
translation invariant filters that are maximal with respect to inclusion,
a simple inductive argument shows
the existence of ``piecewise syndetically"-many
monochromatic arithmetic progressions
of any length $k$ in every finite coloring of the natural numbers.
All the presented constructions are constructive in nature, in the sense that
the involved maximal filters are defined by recurrence on suitable 
countable algebras of sets. No use of the axiom of choice or 
of Zorn's Lemma is needed.
\end{abstract}

\maketitle

\section*{Introduction}
The importance of maximal objects in mathematics
is well-known, starting from the fundamental examples of
maximal ideals in algebra, and of ultrafilters in 
certain areas of topology and of Ramsey theory.
In this paper we focus on maximal filters on suitable 
countable algebras of sets
which are stable under translations.
By using such maximal objects, along with ultrafilters extending it,
we give a proof of a strong version of the following classical
result in Ramsey theory:

\smallskip
\noindent
\textbf{Theorem} (van der Waerden - 1927)
\emph{In every finite partition $\N=C_1\cup\ldots\cup C_r$
there exists a piece $C=C_i$ that contains
arbitrarily long arithmetic progressions,
that is, for every $k$ there exists a progression
$x+y, x+2y,\ldots,x+ky\in C$.}

\smallskip
In fact, we will prove the 
existence of ``piecewise syndetically"-many
monochromatic arithmetic progressions
of any length $k$.

Usually, van der Waerden's Theorem is proved either by double
induction using elementary, but elaborated, combinatorial arguments
in the style of the original proof \cite{vdW},
or by using properties of the smallest ideal 
$K(\beta\N,\oplus)$ in the algebra
of ultrafilters (see \cite[Ch.14]{hs}; see also
\cite{fg,bh} for stronger versions).
In our proof, for any given piecewise set,
we restrict to a suitable countable algebra
of sets, and explicitly construct by recursion a maximal translation invariant
filter, and then an ultrafilter extending it.
The desired result is finally obtained by a short proof by induction,
that is essentially a simplified version of an argument
that was used in \cite{h} in the framework of
the compact right-topological semigroup $(\beta\N,\oplus)$.
It is worth remarking that,
contrarily to the usual ultrafilter proof, we make no explicit use
of the algebra in the space of ultrafilters; in fact,
we make no use of the axiom of choice nor of
Zorn's Lemma.

\section{Preliminary notions}

$\N=\{1,2,3,\ldots\}$ denotes the set of \emph{positive integers}, 
and $\N_0=\N\cup\{0\}$ the set
of non-negative integers.
For $A\subseteq\N$ and $n\in\N_0$, 
the \emph{leftward shift} of $A$ by $n$ is the set:
$$A-n:=\{m\in\N\mid n+m\in A\}$$

Elemental notions in combinatorics of numbers that we will use
in this paper are those of thick set, syndetic set, and piecewise syndetic set.
For completeness, let us recall them here.

A set $A\subseteq\N$ is \emph{thick} if it includes arbitrarily long intervals.
Equivalently, $A$ is thick if every finite set $F=\{n_1,\ldots,n_k\}\subset\N$
has a \emph{rightward shift} included in $A$, that is, there exists $x$ such that
$$F+x:=\{n_1+x,\ldots,n_k+x\}\subseteq A.$$ 
Notice that such an $x$ can
be picked in $A$. In terms of intersections, the property
of thickness of $A$ can be rephrased by saying that
the family $\{A-n\mid n\in\N_0\}$ has the \emph{finite intersection property}
(FIP for short), that is, $\bigcap_{i=1}^k(A-n_i)\ne\emptyset$
for any $n_1,\ldots,n_k$. 

A set $A\subseteq\N$ is \emph{syndetic} if it has ``bounded gaps",
that is, there exists $k\in\N$ such that $A$ meets every
interval of length $k$. Equivalently, $A$ is syndetic if
a finite number of leftward shifts of $A$ covers all the natural numbers,
that is, $\N=\bigcup_{i=1}^k(A-n_i)$ for suitable $n_1,\ldots,n_k\in\N_0$.

A set is \emph{piecewise syndetic} if it is the intersection of a thick set with
a syndetic set. Equivalently, $A$ is piecewise syndetic
if a finite number of leftward shifts cover a thick set, that is,
$\bigcup_{i=1}^k(A-n_i)$ is thick for suitable $n_1,\ldots,n_k\in\N_0$.

Notice that the families of thick, syndetic, and piecewise syndetic sets
are all invariant with respect to shifts.
A well-known relevant property of piecewise syndetic sets that is satisfied
neither by thick sets nor by syndetic sets, is the \emph{Ramsey property} below.
For the sake of completeness, we include here a proof.

\begin{proposition}\label{psRamsey} 
In every finite partition $A=C_1\cup\ldots\cup C_r$ of a piecewise
syndetic set $A$, one of the pieces $C_i$ is piecewise syndetic.
\end{proposition}

\begin{proof}
For simplicity, let us say that an interval $I$ is $k$-good for the set $B$
if for every sub-interval $J\subseteq I$ of length $k$ one has $J\cap B\ne\emptyset$.
By the hypothesis of piecewise syndeticity of $A$,
there exists $k\in\N$ and a sequence of
intervals $\langle I_n\mid n\in\N\rangle$ with
increasing length such that every $I_n$ is $k$-good for $A$.
It is enough to consider the case when $A=C_1\cup C_2$ is partitioned into two pieces,
because the general case $A=C_1\cup\ldots\cup C_r$ where $r\ge 2$ 
will then follow by induction.
We distinguish two cases. 

Case \# 1: There exists $h$ such that infinitely many intervals $I_n$ 
are $h$-good for $C_1$.
In this case $C_1$ is piecewise syndetic.

Case \# 2: For every $h$, there are only finitely many intervals $I_n$ that
are $h$-good for $C_1$. So, for every $h$ we can pick an interval 
$I_{n_h}$ of length $\ge h$ that is not $h$-good.
Let $J_h\subseteq I_{n_h}$ be a sub-interval of length $h$ such
that $J_h\cap C_1=\emptyset$.
The sequence of intervals $\langle J_h\mid h\in\N\rangle$ shows
that $C_2$ is piecewise syndetic.
Indeed, given $h$, for every sub-interval $J\subseteq J_h$ of length $k$
we have that $J\cap C_1\subseteq J_h\cap C_1=\emptyset$; and so
$J\cap C_2=J\cap A\ne\emptyset$, since $J\subseteq I_{n_h}$
and $I_{n_h}$ is $k$-good for $A$.
\end{proof}

\section{Maximal translation invariant filters}\label{sec:3}

In the following, by \emph{family} we mean a nonempty collection of subsets of $\N$.

\begin{definition}
A family $\G$ is \emph{translation invariant}
if $A\in\G\Rightarrow A-1\in\G$
(and hence, $A-n\in\G$ for all $n\in\N_0$).
\end{definition}

An \emph{algebra of sets} (on $\N$) is a family 
that contains $\N$ and is closed under
finite unions, finite intersections, and complements.
The [translation invariant] algebra \emph{generated} by a family $\G$ is the smallest
[translation invariant] algebra of sets that contains $\G$.

\begin{proposition}
If the family $\G$ is countable, then
one can give explicit constructions of both the (countable)
algebra generated by $\G$, and the (countable)
translation invariant algebra generated by $\G$.
\end{proposition}

\begin{proof}
Let $\langle A_n\mid n\in\N\rangle$ be an enumeration of the sets in $\G$,
and let $\langle F_n\mid n\in\N\rangle$ be an enumeration 
of the nonempty finite sets of natural numbers.\footnote
{\emph{E.g.}, if 
$n=\sum_{k=1}^\infty{a_{nk}}{2^{k-1}}$ is written
in binary expansion
where $a_{nk}\in\{0,1\}$, then we can let 
$F_n:=\{k\mid a_{nk}=1\}$.}
For $A\subseteq\N$, denote $A^{+1}=A$ and $A^{-1}=A^c$.
Then the following family $\B_\G$
is the smallest algebra of sets that contains $\G$:
$$\B_\G:=\left\{
\bigcup_{i=1}^t\Big(\bigcap_{k\in F_{n_i}}A_k^{\sigma_i(k)}\Big)\,\,\Big|\,\,
n_1,\ldots,n_t\in\N,\ \sigma_i:F_{n_i}\to\{+1,-1\}\right\}.$$
Notice that if $\G$ is translation invariant, then also $\B_\G$ is translation
invariant. So, the algebra $\B_{\G'}$ generated by 
the family of shifts $\G':=\{A-n\mid A\in\G, n\in\N_0\}$ is the smallest translation invariant
algebra containing $\G$.
\end{proof}

A \emph{filter} on an algebra of sets $\mathcal{B}$
is a nonempty family $\F\subseteq\mathcal{B}$ such that:
\begin {itemize}
\item
$\F$ is closed under finite intersections, that is, 
$A,B\in\F\Rightarrow A\cap B\in\F$;
\item
$\F$ is closed under supersets, that is, 
if $B\in\mathcal{B}$ and $B\supseteq A\in\F$ then $B\in\F$.
\end{itemize}

Every family $\G\subseteq\mathcal{B}$
with the \emph{finite intersection property} (FIP for short) 
generates a filter $\langle\G\rangle$, namely 
$$\langle\G\rangle:=
\{B\in\mathcal{B}\mid B\supseteq A_1\cap\ldots\cap A_k\ \mathrm{for\ suitable}\ 
A_1,\ldots,A_k\in\G\}.$$

An \emph{ultrafilter} $\U$ on the algebra of sets $\B$ is a filter with the additional 
property that $A\in\U$ whenever $A\in\B$ and the complement $A^c\notin\U$.
It is easily verified that a filter $\U$ is an ultrafilter if and only if 
the \emph{Ramsey property} holds:
If $A_1\cup\ldots\cup A_k\in\U$ where all sets $A_i\in\B$,
then $A_j\in\U$ for some $j$.
Ultrafilters can also be characterized as those filters
that are maximal under inclusion and so, by a straight
application of \emph{Zorn's Lemma}, 
it is proved that every filter can be extended to an ultrafilter.

The following objects are the
main ingredient in our proof of van der Waerden's Theorem.

\begin{definition}
A \emph{translation invariant filter} (TIF for short)
is a filter $\F$ on a translation invariant algebra $\B$
such that $A\in\F\Rightarrow A-1\in\F$ (and hence $A-n\in\F$ for all $n\in\N_0$).
\end{definition}

Notice that if the algebra $\B$ is translation invariant,
and the family $\G\subseteq\B$ is translation invariant,
then the generated filter $\langle\G\rangle$ is a TIF.

The notions of TIF and thick set are closely related.

\begin{proposition}\label{thickTIF}
A set $A$ is thick if and only if it belongs to a TIF $\F$.
\end{proposition}

\begin{proof}
Recall that $A$ is thick if and only if the family $\G=\{A-n\mid n\in\N_0\}$
has the FIP. Since $\G$ is translation invariant,
the generated filter $\langle\G\rangle$ is a TIF that contains $A$.

Conversely, assume that $A\in\F$ for some TIF $\F$.
Then trivially the family $\G=\{A-n\mid n\in\N_0\}$ has the FIP because $\G\subseteq\F$.
\end{proof}

Similarly to ultrafilters,
by a straightforward application of \emph{Zorn's Lemma}
it can be shown that every TIF can be extended to a maximal TIF.
However, in the countable case, recursive constructions suffice
to produce both ultrafilters and maximal TIFs, which are thus obtained
in a constructive manner, without any use of the axiom of choice.

\begin{proposition}\label{maximal}
Let $\B=\{B_n\mid n\in\N\}$ be a countable algebra of sets.
\begin{enumerate}
\item
Given a family $\G\subseteq\B$ with the FIP, inductively define 
$\G_0=\G$; $\G_{n+1}=\G_n\cup\{B_n\}$
in case $B_n\cap A\ne\emptyset$ for every $A\in\G_n$; and
$\G_{n+1}=\G_n$ otherwise. Then 
$\U:=\bigcup_n\G_n$ is an ultrafilter on $\B$ that extends $\G$.
\item
Assume that the algebra $\B$ is translation invariant.
Given a 
translation invariant family $\G\subseteq\B$ with the FIP, inductively define 
$\G_0=\G$; $\G_{n+1}=\G_n\cup\{B_n-k\mid k\in\N_0\}$
in case that union has the FIP; and
$\G_{n+1}=\G_n$ otherwise. Then 
$\mathcal{M}:=\bigcup_n\G_n$ is a maximal TIF
that extends $\G$.
\end{enumerate}
\end{proposition}

\begin{proof}
(1). By the definition, it is clear that all families $\G_n$ have the FIP, and 
so also their increasing union $\U$ has the FIP. 
Now assume by contradiction that $A\in\B$ is such that
both $A,A^c\notin\U$. If $A=B_n$ and $A^c=B_m$ then, by the definition
of $\U$, there exist 
$U\in\G_n$ and $U'\in\G_m$ such that
$A\cap U=A^c\cap U'=\emptyset$, and hence
$U\cap U'=\emptyset$, against the FIP of $\U$.
Finally, if $B\supseteq A$ where $B\in\B$ and $A\in\U$ then
$B\in\U$, as otherwise, by what just proved, $B^c\in\U$,
and hence $\emptyset=B^c\cap A\in\U$, a contradiction.

(2). By induction, it directly follows from the definition that all families 
$\G_n$ have the FIP and are translation invariant;
so, the same properties hold for $\mathcal{M}$.
Now let $B\supseteq A$ where $A\in\mathcal{M}$
and $B\in\B$, say $B=B_n$. Notice that $\G_n\cup\{B-k\mid k\in\N_0\}$ has the FIP
because $A-k\subseteq B-k$ for all $k$ and
$\G_n\cup\{A-k\mid k\in\N_0\}\subseteq\mathcal{M}$ has the FIP.
Then $B\in\G_{n+1}\subseteq\mathcal{M}$, and we can conclude that
$\mathcal{M}$ is a TIF.
As for the maximality, let $\mathcal{M}'\supseteq\mathcal{M}$ be a TIF. 
Given $A\in\mathcal{M}'$, pick $n$ with $A=B_n$. The family 
$\G_n\cup\{A-n\mid n\in\N_0\}$ has the FIP, since it
is included in the filter $\mathcal{M}'$, and so $A\in\G_{n+1}$.
This shows that $\mathcal{M}'\subseteq\mathcal{M}$, and hence the
two TIFs are equal.
\end{proof}

Two properties of maximal TIFs that will be relevant to our
purposes are the following.

\begin{proposition}\label{syndetic}
Let $\B$ be a translation invariant algebra, and
let $\U$ be an ultrafilter on $\B$ that includes a maximal TIF $\mathcal{M}$.
Then:
\begin{enumerate}
\item
Every $B\in\U$ is piecewise syndetic.
\item
For every $B\in\U$, the set $B_\U:=\{n\in\N\mid B-n\in\U\}$ is syndetic.\footnote
{~We remark that in general the set $B_\U$ does not belong to the algebra 
of sets $\B$.}

\end{enumerate}
\end{proposition}

\begin{proof}
Notice first that for every $B\in\U$ there exist $n_1,\ldots,n_k$ such that 
the union $\bigcup_{i=1}^k(B-n_i)\in\mathcal{M}$. Indeed, if
$\Lambda:=\{B^c-n\mid n\in\N_0\}$ then the union $\mathcal{M}\cup\Lambda$
does not have the FIP, as otherwise $\mathcal{M}\cup\Lambda$ would generate
a TIF that properly extends $\mathcal{M}$ (since it would contain $B^c$ 
while $B^c\notin\mathcal{M}$), against the maximality. 
So, there exist $A\in\mathcal{M}$ and
$n_1,\ldots,n_k$ such that $A\cap\bigcap_{i=1}^k(B^c-n_i)=\emptyset$.
But then $\bigcup_{i=1}^k(B-n_i)\in\mathcal{M}$, because it
is a superset of $A\in\mathcal{M}$.

(1). Pick a finite union of shifts 
$\bigcup_{i=1}^k(B-n_i)\in\mathcal{M}$. By Proposition \ref{thickTIF},
that union is thick because it is as an element of a TIF, 
and hence $B$ is piecewise syndetic.

(2). As above, pick a finite union of shifts 
$\bigcup_{i=1}^k(B-n_i)\in\mathcal{M}$.
By translation invariance, for every $m\in\N$ one has that
$\bigcup_{i=1}^k(B-n_i-m)\in\mathcal{M}\subseteq\U$
and so, by the Ramsey property of ultrafilters, there exists $i$ such that
$B-n_i-m\in\U$, that is, $m\in B_\U-n_i$.
This shows that $\N=\bigcup_{i=1}^k(B_\U-n_i)$ is a finite union of shifts of $B_\U$,
and hence $B_\U$ is syndetic.
\end{proof}

\section{A strong version of van der Waerden's Theorem}

The following property of piecewise syndetic sets 
was first proved by exploiting the properties
of ultrafilters in the smallest ideal of the
right-topological semigroup $(\beta\N,\oplus)$
(see \cite{fg,bh}).

\begin{theorem}\label{main}
Let $A$ be a piecewise syndetic set. Then for every $k\in\N$, the set 
$\mathrm{AP}_k(A):=
\{x\in A\mid \exists y\in\N\ \mathrm{s.t.}\ x+i y\in A\ \mathrm{for}\ i=1,\ldots,k\}$
is piecewise syndetic.
\end{theorem}

Notice that, as a straight consequence, one obtains the following strong version
of van der Waerden's Theorem.

\begin{theorem}
In every finite partition $\N=C_1\cup\ldots\cup C_r$
there exists a piece $C=C_i$ such that,
for every $k\in\N$, the set 
$\mathrm{AP}_k(C)$ is piecewise syndetic.
\end{theorem}

\begin{proof}
By the Ramsey property of piecewise syndetic sets 
(see Proposition \ref{psRamsey}),
we can pick a color $C_i$ which is piecewise syndetic.
\end{proof}

In this section we will give a new proof of the above theorem
which relies on the existence of an ultrafilter $\U$ on 
the appropriate translation invariant algebra $\B$,
which extends a maximal TIF
and contains a shift of $A$.

\begin{proof}[of Theorem \ref{main}]
Let $\mathcal{B}$ be the (countable)
translation invariant algebra of sets
generated by 
the translation invariant family $\{A-n\mid n\in\N_0\}$.
By the property of piecewise syndeticity, 
a finite union of shifts $T=\bigcup_{j=1}^m(A-n_j)$ is thick.
Then the translation invariant family 
$\G:=\{T-n\mid n\in\N_0\}\subseteq\B$ has the FIP,
and by Proposition \ref{maximal}
we can pick a maximal TIF $\mathcal{M}$ on $\B$ with $\mathcal{M}\supseteq\G$,
and an ultrafilter $\U$ on $\B$ with $\U\supseteq\mathcal{M}$.
The desired result is a consequence of the following general property.

\smallskip
\textbf{Claim.}
\emph{Let $\U$ be an ultrafilter that extends a maximal TIF.
If a shift $B-\ell\in\U$ for some $\ell\in\N_0$,
then $B_\U-\ell$ contains arbitrarily long arithmetic progressions.}

\smallskip
Indeed, since the finite union $T=\bigcup_{j=1}^m(A-n_j)\in\G\subseteq\U$, 
by the Ramsey property of ultrafilters there
exists $n_j$ such that $A-n_j\in\U$.
By the Claim, for every $k\in\N$ there
exist $x$ and $y$ such that $x+i y\in A_\U-n_j$ for $i=0,1,\ldots,k$.
But then $B:=\bigcap_{i=0}^k(A-n_j-x-iy)\in\U$,
and hence also the superset $\mathrm{AP}_k(A)-n_j-x\supseteq B$ belongs to $\U$,
as one can easily verify.
Now recall that all sets in $\U$ are piecewise syndetic
by Proposition \ref{syndetic}, and so we can conclude that
$\mathrm{AP}_k(A)$ is piecewise syndetic because it is a shift
of a member of $\U$.

\smallskip
We are left to prove the Claim.
We proceed by induction on $k$, and prove that if $B-\ell\in\U$ for some 
$\ell\in\N_0$, then $B_\U-\ell$ contains a $k$-term arithmetic progression.\footnote
{~This inductive construction uses a simplified version of an argument 
in \cite{h}.}

If $B-\ell\in\U$, then the set $(B-\ell)_\U=B_\U-\ell$ is syndetic by Proposition \ref{syndetic}.
In particular, $B_\U-\ell\ne\emptyset$, and this proves the induction base $k=1$.

Let us turn to the inductive step $k+1$, and assume that $B-\ell\in\U$.
Let $\ell_0=\ell$. 
By syndeticity of $B_\U-\ell_0$, there exists a finite $F\subset\N_0$ such that
for every $n\in\N$ there exists $x\in F$ with $\ell_0+n+x\in B_\U$.
For convenience, let us assume that $0\in F$.
By the inductive hypothesis, there exist $\ell_1\in\N_0$ and 
$y_1\in\N$ such that $\ell_1+i y_1\in B_\U-\ell_0$ for $i=1,\ldots,k$, that is,
$\ell_0+\ell_1+x_0+i y_1\in B_\U$ where $x_0=0\in F$.
Pick $x_1\in F$ with $\ell_0+\ell_1+x_1\in B_\U$.
If $x_1=x_0$ then we already
found a $(k+1)$-term arithmetic progression in $B_\U-\ell_0$,
as desired. Otherwise, let us consider the intersection
$$B_1:=(B-x_1)\cap\bigcap_{i=1}^k(B-x_0-i y_1).$$
Since $\ell_0+\ell_1+x_1\in B_\U$ and
$\ell_0+\ell_1+x_0+i y_1\in B_\U$ for all $i=1,\ldots,k$, the shift $B_1-\ell_0-\ell_1\in\U$
and so, by the inductive hypothesis, there exist $\ell_2\in\N_0$
and $y_2\in\N$ such that $\ell_2+i y_2\in (B_1)_\U-\ell_0-\ell_1$ for $i=1,\ldots,k$.
In consequence, $\ell_0+\ell_1+\ell_2+x_0+i(y_1+y_2)\in B_\U$
and $\ell_0+\ell_1+\ell_2+x_1+i y_2\in B_\U$ for every $i=1,\ldots,k$.
Pick $x_2\in F$ such that $\ell_0+\ell_1+\ell_2+x_2\in B_\U$.
Notice that if $x_2=x_0$ or $x_2=x_1$ then we have
a $(k+1)$-term arithmetic progression in $B_\U-\ell_0$. Otherwise, let
us consider the intersection
$$B_2:=(B-x_2)\cap\bigcap_{i=1}^k(B-x_1-iy_2)\cap
\bigcap_{i=1}^k(B-x_0-i(y_1+y_2)).$$
Similarly as above, one can easily verify that
$B_2-\ell_0-\ell_1-\ell_2\in\U$ and so, by the inductive hypothesis,
we can pick an arithmetic progression in $B_\U-\ell_0-\ell_1-\ell_2$
of length $k$. We iterate the procedure.
As the set $F$ is finite,
after finitely many steps we will find elements $x_n=x_m$ where $n>m$,
and finally obtain the following arithmetic progression of length $k+1$:
$$\ell_0+\ell_1+\ldots+\ell_n+x_n+i(y_{m+1}+\ldots+y_n)\quad
i=0,1,\ldots,k.$$
\end{proof}

\section{TIFs and left ideals in the space of ultrafilters}

The usual ultrafilter proof of van der Waerden's Theorem
(see \cite[\S 14.1]{hs})
is grounded on the existence of minimal ultrafilters, that is, 
on those ultrafilters that belong to some minimal left ideals of the
compact right-topological semigroup $(\beta\N,\oplus)$.
In this final section, we show how
(maximal) translation invariant filters are in fact related
to the closed (minimal) left ideals of $(\beta\N,\oplus)$. 
Let us recall here the involved notions.

The \emph{space $\beta\N$} is the topological space of all
ultrafilters $\U$ over the full algebra of sets $\B=\mathcal{P}(\N)$ where
a base of (cl)open sets is given by the family
$\{\mathcal{O}_A\mid A\subseteq\N\}$, with
$\mathcal{O}_A:=\{\U\in\beta\N\mid A\in\U\}$.
The space $\beta\N$ is Hausdorff and compact, and coincides with
the \emph{Stone-C\u{e}ch compactification} of the discrete space $\N$.

The \emph{pseudosum} $\U\oplus\V$ of ultrafilters $\U,\V\in\beta\N$ is defined by letting:
$$A\in\U\oplus\V\ \Longleftrightarrow\ \{n\in\N\mid A-n\in\V\}\in\U.$$
The operation $\oplus$ is associative (but not commutative),
and for every $\V$ the map $\U\mapsto\U\oplus\V$ is continuous.
This makes $(\beta\N,\oplus)$ a \emph{right-topological semigroup}.

A \emph{left ideal} $L\subseteq\beta\N$ is a nonempty set
such that $\V\in L$ implies $\U\oplus\V\in L$ for all $\U\in\beta\N$.
The notion of \emph{right ideal} is defined similarly.
Left ideals that are minimal with respect to inclusion are particularly
relevant objects, as they satisfy special properties. 
For instance, their union $K(\beta\N,\oplus)$
is shown to be the smallest bilater ideal (\emph{i.e.} it is both a left and a right ideal).
Moreover, all ultrafilters $\U$ in $K(\beta\N,\oplus)$, named \emph{minimal ultrafilters},
have the property
that every set $A\in\U$ includes arbitrarily long arithmetic progressions.\footnote
{~For all notions and basic results on the space of ultrafilters $\beta\N$
and on its algebraic structure, including properties of the smallest
ideal $K(\beta\N,\oplus)$, we refer the reader to the book \cite{hs}.}

It is well-known that there are natural correspondences 
between families with the finite intersection
property on the full algebra $\mathcal{P}(\N)$, 
and closed nonempty subsets of $\beta\N$.
Indeed, the following properties are directly verified from the definitions.

\begin{itemize}
\item
If $\G\subseteq\mathcal{P}(\N)$ is a family with the FIP 
then $\mathfrak{C}(\G):=\{\V\in\beta\N\mid\V\supseteq\G\}$
is a nonempty closed subspace.
\item
If $X\subseteq\beta\N$ is nonempty then
$\mathfrak{F}(X):=\bigcap\{\V\mid \V\in X\}$
is a filter on $\mathcal{P}(\N)$.
\item
$\mathfrak{C}(\mathfrak{F}(X))=\overline{X}$ (the topological closure of $X$)
for every nonempty $X\subseteq\beta\N$. 
\item
$\mathfrak{F}(\mathfrak{C}(\G))=\langle\G\rangle$ (the filter generated by $\G$)
for every family $\G\subseteq\mathcal{P}(\N)$ with the FIP.
\end{itemize}

\begin{proposition}
If $\F$ is a TIF on $\mathcal{P}(\N)$ then 
$\mathfrak{C}(\F)$ is a closed left ideal of $(\beta\N,\oplus)$; and conversely, 
if $L$ is a left ideal of $(\beta\N,\oplus)$ then
$\mathfrak{F}(L)$ is a TIF on $\mathcal{P}(\N)$. 
Moreover, $\mathcal{M}$ is a maximal TIF on $\mathcal{P}(\N)$
if and only if $\mathfrak{C}(\mathcal{M})$
is a minimal left ideal of $(\beta\N,\oplus)$; and
$L$ is a minimal left ideal of $(\beta\N,\oplus)$
if and only if $\mathfrak{F}(L)$ is a maximal TIF on $\mathcal{P}(\N)$.
\end{proposition}

\begin{proof}
Let $\V\in\mathfrak{C}(\F)$ and let $\U\in\beta\N$ be any ultrafilter.
For every $A\in\F$, by translation invariance we know that $A-n\in\F$ for all $n$,
and so $\{n\mid A-n\in \V\}=\N\in\U$.
This shows that $A\in\U\oplus\V$. As this is true for every $A\in\F$,
we conclude that $\U\oplus\V\in\mathfrak{C}(\F)$, and so $\mathfrak{C}(\F)$
is a closed left ideal.

Now let $L$ be a left ideal, and let $A\in\mathfrak{F}(L)$ be
in the filter determined by $L$. For every $\V\in L$,
we have that $\mathfrak{U}_1\oplus\V\in L$,
where $\mathfrak{U}_1:=\{B\subseteq\N\mid 1\in B\}$ is the principal
ultrafilter generated by $1$. Then
$A\in\mathfrak{U}_1\oplus\V$, which is equivalent to $A-1\in\V$.
As this holds for every $\V\in L$, we have proved that
$A-1\in\mathfrak{F}(L)$, and so $\mathfrak{F}(L)$ is a TIF, as desired.

Let $\F$ be a TIF. If the left ideal $\mathfrak{C}(\F)$ is not minimal,
pick a minimal $L\subsetneq\mathfrak{C}(\F)$.
Then $\F\subsetneq \mathfrak{F}(L)$, and hence $\F$ is not maximal.
Indeed, $L\subseteq \mathfrak{C}(\F)\Rightarrow \mathfrak{F}(L)\supseteq
\mathfrak{F}(\mathfrak{C}(\F))=\F$;
moreover, $\F\ne\mathfrak{F}(L)$, as otherwise
$\mathfrak{C}(\F)=\mathfrak{C}(\mathfrak{F}(L))=\overline{L}=L$,
against our assumptions.
(Recall that a minimal left ideal
$L$ is necessarily closed because,
by minimality, $L=\beta\N\oplus\V:=\{\U\oplus\V\mid \U\in\beta\N\}$ 
for every given $\V\in L$, and $\beta\N\oplus\V$ is closed as it the image 
of the compact Hausdorff space $\beta\N$ under 
the continuous function $\U\mapsto\U\oplus\V$.)
In a similar way, one shows the converse implication:
If the TIF $\F$ is not maximal then the left ideal $\mathfrak{C}(\F)$ is not minimal.
In consequence, $L=\mathfrak{C}(\mathfrak{F}(L))$ is minimal
if and only if $\mathfrak{F}(L)$ is maximal, and also
the last equivalence is proved.
\end{proof}

As a straight consequence, we obtain the desired characterization.

\begin{proposition}
An ultrafilter $\U$ on $\mathcal{P}(\N)$ includes a maximal TIF if and only if
$\U$ belongs to the smallest ideal $K(\beta\N,\oplus)$.
\end{proposition}

\begin{proof}
Recall that $\U\in K(\beta\N,\oplus)$ if and only if
$\U$ belongs to some minimal left ideal.
Now let $\U\supseteq\mathcal{M}$ where $\mathcal{M}$ is
a maximal TIF. Since $\mathcal{M}=\mathfrak{F}(\mathfrak{C}(\mathcal{M}))$, 
we have that $\U\in\mathfrak{C}(\mathcal{M})$,
where $\mathfrak{C}(\mathcal{M})$ is a minimal left ideal.
Conversely, let $\U\in L$ where
$L$ is a minimal left ideal. Then $\mathfrak{F}(L)$ is a maximal TIF
and $\U\supseteq\mathfrak{F}(L)$, since
$\U\in L=\mathfrak{C}(\mathfrak{F}(L))$.
\end{proof}

\begin{remark}
One can generalize the contents of this paper
from the natural numbers to arbitrary countable semigroups $(S,\cdot)$.
Indeed, the notion of translation invariant filter also
makes sense in that more general framework.\footnote
{~Actually, our techniques also apply for uncountable semigroups,
but in that case one needs Zorn's Lemma
to prove the existence of maximal TIFs and of ultrafilters.}
Precisely, for $A\subseteq S$ and $s\in S$,
denote by $s^{-1}A:=\{t\in S\mid s\cdot t\in A\}$.
We say that an algebra $\B$ of subsets of $S$ is 
\emph{translation invariant} if $B\in\B\Rightarrow s^{-1}B\in\B$ for all $s\in S$.
Then one defines a TIF on a translation invariant algebra $\B$ 
as a filter $\F$ such that $A\in\F\Rightarrow s^{-1}A\in\F$
for all $s\in S$.
By the same arguments as the ones used in this paper,
one can prove that a reformulation of Theorem \ref{main} holds,
provided one adopts the appropriate generalization of the
notion of piecewise syndetic set.\footnote
{~In an arbitrary semigroup $(S,\cdot)$, one
defines a subset $T\subseteq S$ to be \emph{thick} if
for every finite $F$ there exists $s\in S$ with 
$F\cdot s:=\{x\cdot s\mid x\in F\}\subseteq T$;
a set $A\subseteq S$ is \emph{syndetic} if a suitable finite union 
$\bigcup_{i=1}^k s_i^{-1}A=S$ 
covers the whole semigroup; and finally a set $A\subseteq S$ is
\emph{piecewise syndetic}
if a suitable finite union $\bigcup_{i=1}^k s_i^{-1}A$ is thick
(see \cite[\S 4.4 and \S 4.5]{hs}).
}
\end{remark}

\smallskip
\textbf{Acknowledgement.}
I would like to thank the anonymous referee for carefully reading the first version of this
paper and for giving comments that were helpful for the final revision.

\end{document}